\documentclass[a4paper]{amsart}

\usepackage[english]{babel} 
\usepackage[utf8x]{inputenc}
\usepackage[T1]{fontenc}

\usepackage{amsmath,amsthm,amsfonts,amstext,amssymb}
\usepackage{lmodern}
\usepackage{graphicx}
\usepackage[colorlinks=true, allcolors=blue]{hyperref}
\usepackage{xcolor} 
\usepackage{fancyhdr}
\usepackage{mathtools}
\usepackage{caption}
\usepackage{subcaption}
\usepackage[export]{adjustbox}
\numberwithin{equation}{section}

\newcommand{\m}{\overline{m}}
\newcommand{\s}{\mathrm{sgn}}
\newcommand{\id}{\mathrm{id}}
\newcommand{\defi}{\overset{\mathrm{def}}=}

\makeatletter
\theoremstyle{plain}
\newtheorem{thm}{\protect Theorem}[section]
\newtheorem{lem}[thm]{\protect Lemma}
\newtheorem{prop}[thm]{\protect Proposition}
\newtheorem{conj}[thm]{\protect Conjecture}
\newtheorem{cor}[thm]{\protect Corollary}
\newtheorem{obs}[thm]{\protect Observation}
\theoremstyle{definition}

\theoremstyle{remark}
\newtheorem{rem}[thm]{\protect Remark}

\newtheorem*{claimno}{\protect Claim}
\makeatother

\usepackage{ifpdf} 
\ifpdf 
 \IfFileExists{lmodern.sty}{\usepackage{lmodern}}{}
\fi 

\begin{document}

\lhead{Standard Identities on $M_{n}E^{2}$ and $M_{n}E^{3}$}
\rhead{B. A. Bal\'azs, Sz. M\'esz\'aros}
\chead{}

\title{
The Minimal Degree Standard Identity \\on $M_{n}E^{2}$ and $M_{n}E^{3}$}
\author{Barbara Anna Bal\'azs, Szabolcs M\'esz\'aros}

\begin{abstract}
We prove an Amitsur--Levitzki-type theorem for Grassmann algebras, stating that the minimal degree of a standard identity that is a polynomial identity of the ring of  $n \times n$ matrices over the $m$-generated Grassmann algebra is at least $2\left\lfloor\frac{m}{2}\right\rfloor+4n-4$ for all $n,m\geq 2$ and this bound is sharp for $m=2,3$ and any $n\geq 2$. The arguments are purely combinatorial, based on computing sums of signs corresponding to Eulerian trails in directed graphs.
\end{abstract}



\thanks{This research was partially supported by National Research, Development and Innovation Office, NKFIH K 119934 and it is connected to the scientific program of the "Development of 
	quality-oriented and harmonized R+D+I strategy and functional model at 
	BME" project, supported by the New Hungary Development Plan (Project ID: 
	TÁMOP-4.2.1/B-09/1/KMR-2010-0002).}

\address{Department of Computer Science and Information Theory, Budapest University of Technology and Economics, 1111 Budapest, Műegyetem rkp. 3, Hungary}
\email{bbarbara@cs.bme.hu}

\address{MTA Rényi Institute, 1053 Budapest, Reáltanoda utca 13-15, Hungary}
\email{meszaros.szabolcs@renyi.mta.hu}

\maketitle

\section{Introduction}\label{sec:intro}

In \cite{S1,S2} R. G. Swan gave a graph-theoretic proof of the Amitsur--Levitzki theorem
which states that the standard identity of degree $2n$ holds for the ring of $n\times n$ matrices over a commutative ring. We generalize these methods and extend the Amitsur--Levitzki theorem to the case where the commutative ring is replaced by a finite dimensional Grassmann algebra.

Let $R$ be a commutative unital ring, $m,n\geq 1$ and denote by $M_{n}E^{m}$ the ring of $n\times n$-matrices over the $m$-generated Grassmann algebra 
$$E^{m}\overset{\mathrm{def}}{=}R\langle v_{i}\ |\ 1\le i\leq m\rangle/(v_{i}^{2},v_{i}v_{j}+v_{j}v_{i}\, \mid 1 \leq i,j \leq m).$$
We say that the standard identity of degree $k$ is a polynomial identity of $M_{n}E^{m}$ if and only if for any $x_{1},\dots,x_{k}\in M_{n}E^{m}$

\[
s_{k}(x_{1},\dots,x_{k})\overset{\mathrm{def}}{=}
\sum_{\pi\in\mathfrak{S}_{k}}\mathrm{sgn}(\pi)
x_{\pi(1)}x_{\pi(2)}\dots x_{\pi(k)}=0\]
where $\mathfrak{S}_{k}$ denotes the symmetric group on the set $\{1,\dots,k\}$. 

The problem of finding the smallest $k$ such that the standard identity of degree $k$ holds for $M_nE^m$ was raised in \cite{MMSzW}. They proved an upper bound for this quantity applying a theorem of M. Domokos (see \cite{D}), where a connection between signed sums of permutations corresponding to Eulerian trails and standard identities of matrix algebras is spelled out. This upper bound was improved in \cite[Prop. 8]{F}, where it was also asked whether the standard identity of degree $2\left(n+\left\lfloor \frac m2 \right\rfloor\right)$ holds in $M_nE^m$. For $m=0,1$ the exact value of the smallest $k$ is $2n$ by the Amitsur--Levitzki theorem (see \cite{AL}). In this paper we prove the following:

\begin{thm}\label{thm:lower}
The standard identity of degree $2\left\lfloor\frac{m}{2}\right\rfloor+4n-5$ is not a polynomial identity of $M_nE^m$ for any $n,m\geq 2$. 
\end{thm}

\begin{thm}\label{thm:upper}
The standard identity of degree $4n-2$ is a polynomial identity of $M_nE^2$ and $M_nE^3$ for all $n\geq 2$.
\end{thm}

In particular, the case of $m=2,3$ is settled.

The construction we use in the proof of the lower bound is new. The argument showing the upper bound is a substantial enhancement of the methods used in \cite{S1}. Based on the results we conjecture that

\begin{conj}
	For all $n,m\geq 2$ the minimal degree of a standard identity that is a polynomial identity of $M_{n}E^{m}$ is $2\left\lfloor\frac{m}{2}\right\rfloor+4n-4$.
\end{conj}

The article is organized as follows: In Section \ref{sec:reformulation} we reformulate the problem in terms of directed Eulerian trails.  In Section \ref{sec:lowerbound} we prove Theorem \ref{thm:lower}, while in Section \ref{sec:upperbound} we show Theorem \ref{thm:upper}. In Appendix \ref{appendix}, we complete the proof of \mbox{Prop. \ref{prop:equivalence}}.

\section{Graph-theoretic reformulation}\label{sec:reformulation}

Following the ideas of \cite{S1,S2} we formulate the graph-theoretic
counterpart of the problem of standard identities on $M_{n}E^{m}$. 

Let $G=(V,A,s,t)$ be a doubly-rooted directed graph (possibly with loops and multiple edges), that is, each edge $a\in A$ has a unique source $a^-$ and target $a^+$ in $V$, moreover $(s,t)\in V\times V$ is a fixed pair of vertices called roots.

For any $B\subseteq A$ we will define $T(G,B)\in \mathbb{N}$ such that the following holds:

\begin{prop}\label{prop:equivalence}
For any $n,m,k\in \mathbb{N}^+$, $M_{n}E^{m}$ 
satisfies the standard identity of degree $k$ if 
and only if $T(G,B)=0$ for any doubly-rooted digraph 
$G=(V,A,s,t)$ with $|V|=n$, $|A|=k$ and
$B\subseteq A$ of size at most $m$.
\end{prop}

In short, the idea behind the proposition is the following: consider a $k$-element subset $\{x_1, \dots, x_k\}\subseteq M_n E^m$ where $x_j$ is of the form $v_i E_ {\alpha\beta}$ or $E_ {\alpha\beta}$ (the latter denoting a matrix unit). Then we may associate to it a pair $(G,B)$ such that -- for each fixed pair of indices $s,t\in [n]\overset{\mathrm{def}}{=}\{1,2,\dots,n\}$ -- the matrix entry $(s_k(x_1,\dots,x_k))_{s,t}$ is expressed as $\pm T(G,B)$ times a monomial in $v_1,\dots, v_k$.
The proof of Prop. \ref{prop:equivalence} is given in the Appendix.

First we define a map of permutations $\sigma \mapsto \sigma_M$  
encapsulating the signs in the summands of $s_k$ given by the anti-commutativity of the generators of $E^m$.

For any $k,m \in \mathbb{N}^{+}$ let $M\subseteq [k]$ of size $m$, considered as a totally ordered set with the inherited ordering. For a permutation $\sigma\in\mathfrak{S}_{k}$  define
\begin{equation}\label{sigma_M}
\sigma_{M} \overset{\mathrm{def}}{=} p_{\sigma} \circ \sigma|_{M} \circ j \in \mathfrak{S}_{m}
\end{equation}
where $j:[m]\to M$ is the unique order-preserving bijection and $p_{\sigma}:[k]\to [m]$ is any monotonically increasing function that makes the composition $\sigma_M:[m]\to [m]$ bijective.
Explicitly, if 
$M=\{j_1,\dots,j_{m}\}$
then 
\[
\sigma_M(g)=|\{h \,\mid\, 1\leq h\leq m ,\ 
\sigma(j_{h})\leq \sigma(j_{g})\}|
\]
for all $1\leq g\leq m$. For example, if $m=2$ then $\sigma_M=\mathrm{id}\in \mathfrak{S}_2$ if and only if $\sigma(j_1) < \sigma(j_2)$.
Note that $\sigma \mapsto \sigma_M$ is not a group homomorphism, which may partially explain the complexity of the problem in Theorem \ref{thm:lower} and \ref{thm:upper}.

For any $P\subseteq \mathfrak{S}_k$ and $M\subseteq [k]$ let us define
\begin{equation}\label{eq:s(P,M)}
s(P,M)\overset{\mathrm{def}}{=}
\sum_{\sigma\in P}\mathrm{sgn}(\sigma)\,\mathrm{sgn}(\sigma_M)
\end{equation}
which will help us to express the entries of the matrix $s_k(x_1,\dots,x_k)\in M_{n}E^{m}$.

Let $G=(V,A,s,t)$ be a doubly-rooted digraph. Given an enumeration $A=\{a_{1},\dots,a_{k}\}$ of the edges of $G$ and a subset $B\subseteq A$ define
\begin{equation}\label{eq:M(B)}
M(B)\overset{\mathrm{def}}{=}\{j\in [k]\;|\;a_j\in B\}\qquad\textrm{and}
\end{equation}
\begin{equation}\label{eq:P(G)}
P(G)\overset{\mathrm{def}}{=}
\Big\{\sigma \in \mathfrak{S}_k\;\Big|\; 
a_{\sigma(1)}^-=s,\;a_{\sigma(k)}^+=t,\;
a_{\sigma(j+1)}^- = a_{\sigma(j)}^+\;
(\forall j\in [k-1])\Big\}.
\end{equation}
In short, the set of $\sigma$'s such that $(a_{\sigma(1)},\dots,a_{\sigma(k)})$ is an Eulerian trail from $s$ to $t$.

Given an enumeration of $A$ we define
\[
S(G,B)\overset{\mathrm{def}}{=}
s(P(G),M(B))=
\sum_{\sigma\in P(G)}\mathrm{sgn}(\sigma)\,\mathrm{sgn}(\sigma_{M(B)}).
\]

Note that the absolute value of $S(G,B)$ is independent of the enumeration of the edges.
Indeed, beside the original enumeration, consider the enumeration $a_{\pi(1)},\dots,a_{\pi(k)}$ for some $\pi\in \mathfrak{S}_k$. Suppressing $G$ and $B$ we denote the corresponding sets and maps as 
$M^{\mathrm{id}}$, $P^{\mathrm{id}}$, $\sigma\mapsto \sigma_{M^{\mathrm{id}}}^{\mathrm{id}}$ and 
$M^{\pi}$, $P^{\pi}$, $\sigma\mapsto \sigma_{M^{\pi}}^{\pi}$ respectively. 
Then $M^{\pi}=\pi^{-1}M^{\mathrm{id}}$,
$P^{\pi}=\pi^{-1} P^{\mathrm{id}}$ and
$\sigma_{M^{\pi}}^{\pi} = (\pi \sigma)_{M^{\mathrm{id}}}^{\mathrm{id}}$. Hence 
\[
\sum_{\sigma\in P^{\pi}}\mathrm{sgn}(\sigma)\mathrm{sgn}(\sigma_{M^{\pi}}^{\pi})=
\sum_{\sigma\in \pi^{-1} P^{\mathrm{id}}}\mathrm{sgn}(\sigma)\mathrm{sgn}((\pi\sigma)_{M^{\mathrm{id}}}^{\mathrm{id}})
\]
so -- by changing the index of summation -- we obtain that the value of $S(G,B)$ for the two enumerations differs only by a factor of $\mathrm{sgn}(\pi)$.

Therefore we may define
\[T(G,B)\overset{\mathrm{def}}{=}|S(G,B)|\]
independently of the enumeration.

\begin{rem}\label{rem:reform}
In the terminology of the section we may express the known results. 

\begin{enumerate}
\item The Amitsur--Levitzki theorem is equivalent to the fact that for any doubly-rooted digraph $G$ on $n$ vertices with $2n$ edges we have
$T(G,\emptyset)=0$.

\item In \cite[Cor. 7, Prop. 8]{F} it is shown that for any doubly-rooted digraph $G$ on $n$ vertices with
$2n\left(\lfloor \frac{m}{2}\rfloor+1\right)$ edges 
(resp. $2\left(\left\lfloor\frac{n^2+1}{2}\right\rfloor+\left\lfloor\frac{m}{2}\right\rfloor\right)$ edges) and for a subset of edges $B$ such that $|B|=m$ we have $T(G,B)=0$.

\item Note that the previous two points also hold if the graphs have more than the given number of edges. Generally, if $T(G,B)=0$ for all $G$ on $n$ vertices with $k$ edges and $|B|=m$, then the same holds for all graphs on $n$ vertices with at least $k$ edges and $|B|=m$. Indeed, by Prop. \ref{prop:equivalence} they are equivalent to $s_k$ (resp. $s_l$ for $l\geq k$) being a polynomial identity on $M_nE^m$.

\item The anti-symmetric property of the standard identity shows that if $a,b\in A\backslash B$ are distinct edges that are parallel i.e. $a^-=b^-$ and $a^+=b^+$ then $T(G,B)=0$.
\end{enumerate}
\end{rem}

\section{Lower bound}\label{sec:lowerbound}

In this section we prove Theorem \ref{thm:lower} i.e. we show that the standard identity of degree $2\left\lfloor\frac m2\right\rfloor+4n-5$ is not a polynomial identity of $M_nE^m$ for any $n,m\geq 2$. We use the graph-theoretic reformulation, hence by Prop. \ref{prop:equivalence} it is enough to prove that there exists a doubly-rooted digraph $G$ with $n\geq 2$ vertices, $2\left\lfloor\frac m2\right\rfloor+4n-5$ edges and a subset of the edges $B$ with size $2\left\lfloor\frac m2\right\rfloor$ such that $T(G,B)\neq 0$ holds.

We will consider $2\m\overset{\mathrm{def}}{=}m$ fixed and even throughout the section since $2\left\lfloor\frac m2\right\rfloor$ gives the same value for $m=2\m$ and $m=2\m+1$ and the equivalence in Prop. \ref{prop:equivalence} says $|B|\leq m$.

For each integer $n\geq 2$ consider the doubly-rooted digraph $G_n=([n],A,1,2)$ with the following edges:
$$\begin{array}{ll}
a_{2\ell-1}=(1,2)& a_{m+4h-9}=(h-1,h)\\
a_{2\ell} \,\,\,\,\,\,=(2,1)& a_{m+4h-8}=(h,1)\\
 &a_{m+4h-7}=(1,h)\\
a_{m+1}=(1,1) &a_{m+4h-6}=(h,h-1)\\
a_{m+2}=(1,2)&a_{m+4n-5}=(n,n)
\end{array}$$
for $\ell=1,...,\overline{m}$ and $h=3,...,n$, where $a_{x}=(i,j)$ means that $a_{x}^-=i$ and $a_{x}^+=j$ for all $x\in[m+4n-5].$

\begin{figure}[h!]	\begin{center}	\includegraphics{ch3_figures-1.mps}	\end{center} 	\caption{$G_n$}\end{figure}

\begin{thm}\label{thm:Sn}
	Let $B=\{a_1,a_2,...,a_m\}$ and $(a_1, a_2, ..., a_{m+4n-5})$ be a fixed order of the edges. Then
	\[S\big(G_n,B\big)= \begin{cases}
	\big((\m+1)!\big)^2 &\textrm{if }n=2,\\[1mm]
	\frac23\m(\m+n-1)!\m! &\textrm{if }n\geq 3.
	\end{cases}\]
\end{thm}

In the next lemma we give a recursive formula for $n\geq 4$, so we may prove the theorem by induction. The cases $n=2$ and $n=3$ are managed separately in Lemma \ref{lem:S2} and \ref{lem:S3}. For this section we assume that the fixed order of the edges of $G_n$ is always $(a_1, a_2, ..., a_{m+4n-5})$ and that $B\subseteq A(G_n)$ is the subset $\{a_1,a_2,...,a_m\}$.

\begin{lem}\label{lem:recursive}
	For $n\geq 4$	$$S(G_n,B)=(\m+n-1) S(G_{n-1},B).$$
\end{lem}

Before proving the lemma we make some observations.
\subsection{Notations}\label{Notations}

Let $G=([n],A,1,2)$ be a doubly-rooted digraph, $B\subseteq A$ and given an enumeration $A=\{a_1,a_2,...,a_k\}$. Let $\sigma\in P(G)$ i.e. $(a_{\sigma{(1)}}, a_{\sigma{(2)}}, ..., a_{\sigma{(k)}})$ is an Eulerian trail from $1$ to $2$ and let
$$\begin{array}{ll}
q_1=(a_{\sigma(i)},a_{\sigma(i+1)},...,a_{\sigma(i+r)} ),&1\leq i\leq i+r\leq k\\
q_2=(a_{\sigma(j)},a_{\sigma(j+1)},...,a_{\sigma(j+u)} ),&1\leq j\leq j+u\leq k\end{array}$$
two edge-disjoint subtrails in $G$. We denote by $|q|$ the \textit{length} of the trail $q$, so $|q_1|=r+1$, $|q_2|=u+1$ and by $|q|'$ the number of the edges in the trail $q$ that belong to $B$.

Let us notice that the set of those permutations in $P(G)$ such that $q_1, q_2$ are subtrails, can be partitioned into two subsets according to whether \textit{$q_1$ precedes $q_2$} (i.e. $i+r<j$ and this will be denoted by $q_1<q_2$) or $q_2$ precedes $q_1$:
$$P_{q_1,q_2}(G)=P_{q_1<q_2}(G)\sqcup P_{q_2<q_1}(G)$$
where
$$\begin{array}{l}P_{q_1,q_2}(G)\,\,\overset{\mathrm{def}}=\{\sigma\in P(G)\, \big| \,q_1,q_2 \textrm{ are subtrails}\},\\
P_{q_1<q_2}(G)\overset{\mathrm{def}}=\{\sigma\in P(G)\, \big| \,q_1,q_2 \textrm{ are subtrails, } q_1 \textrm{ precedes } q_2\}.\end{array}$$
Let $M(B)$ be the set defined in Eq. \ref{eq:M(B)} and $s(M,P)$ be the sum defined in \mbox{Eq. \ref{eq:s(P,M)}. Then}
$$s\left(P_{q_1,q_2}(G),M(B)\right)=s\left(P_{q_1<q_2}(G),M(B)\right)+s\left(P_{q_2<q_1}(G),M(B)\right)$$
holds by definition. In addition, if $q_1$, $q_2$ are subtrails in every Eulerian trail then
\begin{equation}\label{eq:S(G,B)}S(G,B)=s\left(P_{q_1,q_2}(G),M(B)\right).\end{equation}

We say that $q_1$ and $q_2$ are \textit{parallel} if $$a_{\sigma(i)}^-=a_{\sigma(j)}^-,\quad a_{\sigma(i+r)}^+=a_{\sigma(j+u)}^+$$ which means that $q_1$ and $q_2$ have the same starting and ending node.

The swap of two parallel subtrails $q_1$ and $q_2$ in an Eulerian trail also results an Eulerian trail, hence the elements of the two subsets defined above can be paired according to the position of $q_1$ and $q_2$ compared to each other. Let
$$\sigma^{q_1\leftrightarrow q_2}\in P_{q_1,q_2}(G)$$
denote the permutation obtained from $\sigma\in P_{q_1,q_2}(G)$ by swapping the parallel trails $q_1$, $q_2$.
Then $\sigma\mapsto \sigma^{q_1\leftrightarrow q_2}$ defines a bijection between $P_{q_1<q_2}(G)$ and $P_{q_2<q_1}(G)$, thus recalling the definition of $s(P,M)$ from Eq. \ref{eq:s(P,M)},
\begin{multline*}
s\left(P_{q_1,q_2}(G),M(B)\right)=\sum\limits_{\sigma\in P_{q_1,q_2}(G)}\mathrm{sgn}(\sigma)\mathrm{sgn}(\sigma_{M(B)})=\\
$$=\sum\limits_{\sigma\in P_{q_1<q_2}(G)}
\left( \mathrm{sgn}(\sigma)\mathrm{sgn}(\sigma_{M(B)})+ \mathrm{sgn}(\sigma^{q_1\leftrightarrow q_2})\mathrm{sgn}(\sigma^{q_1\leftrightarrow q_2}_{M(B)})\right).
\end{multline*}

\begin{obs}\label{obs} Let $\sigma\in P_{q_1,q_2}(G)$ where $q_1, q_2$ are parallel subtrails.
\begin{enumerate}
\item Denote by $h$ the number of edges between $q_1$ and $q_2$. Then
$$ \mathrm{sgn}(\sigma^{q_1\leftrightarrow q_2})= (-1)^{|q_1|(h+|q_2|)+h|q_2|}\mathrm{sgn}(\sigma).$$
\item Denote by $h'$ the number of edges in the trail between $q_1$ and $q_2$ that belong to the set $B$. Then
$$ \mathrm{sgn}(\sigma_M^{q_1\leftrightarrow q_2})= (-1)^{|q_1|'(h'+|q_2|')+h'|q_2|'}\mathrm{sgn}(\sigma_M).$$
\end{enumerate}
\end{obs}

\noindent Let us emphasize the most frequently used cases.

\begin{cor}\label{obs2} Let $q_1, q_2$ be parallel subtrails in the digraph $G$.
\begin{enumerate}
\item If $|q_1|$, $|q_2|$ are odd and $|q_1|'$, $|q_2|'$ are even then
$s\left(P_{q_1,q_2}(G),M(B)\right)=0.$
\item If $|q_1|$, $|q_2|$ are even and $|q_1|'$, $|q_2|'$ are odd then
$s\left(P_{q_1,q_2}(G),M(B)\right)=0.$
\end{enumerate}
\end{cor}

In the end of this subsection we introduce one more definition. We call a map $\sigma\mapsto\pi$ \textit{sign-preserving} (resp. \textit{sign-reversing}) if $\mathrm{sgn}(\pi)\mathrm{sgn}(\pi_{M})=\mathrm{sgn}(\sigma)\mathrm{sgn}(\sigma_{M})$ (resp. $\mathrm{sgn}(\pi)\mathrm{sgn}(\pi_{M})=-\mathrm{sgn}(\sigma)\mathrm{sgn}(\sigma_{M})$), where $M$ will be given by the context.

\subsection{The inductive step}

The idea of the proof of the recursive formula is to partition the set of the Eulerian trails of $G_n$. We choose the partition so that the equation $s(P_i,M(B))=0$ holds for some $P_i$ subsets of the permutations related to the corresponding Eulerian trails. For the remaining Eulerian trails, $s(P,M(B))$ can be easily obtained from the Eulerian trails of $G_{n-1}$.

Throughout the subsection we use the notation $N\defi m+4n$.

\begin{proof}[Proof of Lemma \ref{lem:recursive}]
	One of our major observations is that we can split the node $n$ such that for the resulting graphs $G_n^1$, $G_n^2$, $G_n^3$ and $G_n^4$ the following holds:
\begin{equation}\label{Sn}S(G_n,B)=\sum\limits_{i=1}^{4}S(G_n^i,B).\end{equation}
	 
\noindent The graphs $G_n^1$, $G_n^2$, $G_n^3$ and $G_n^4$ are defined as in Fig. \ref{4graf} with same edge-enumeration as $G_n$.
\begin{figure}[!h]
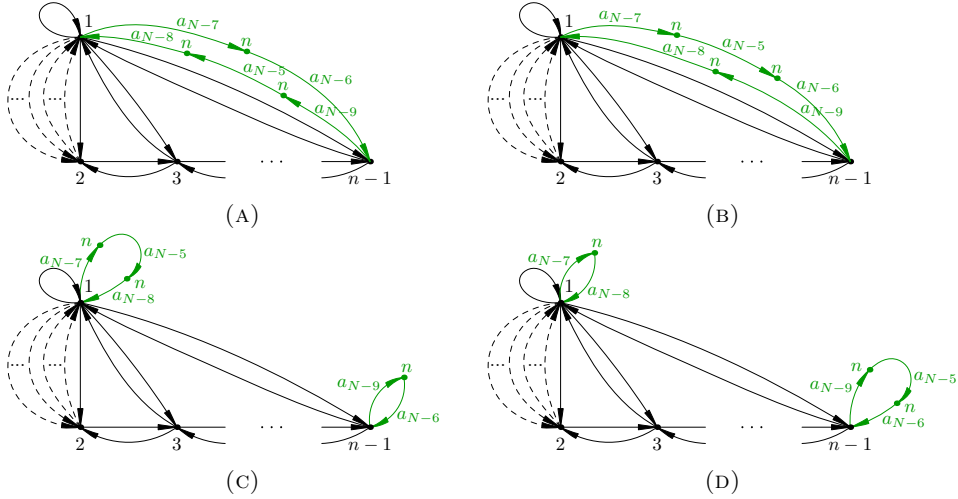

	\centering
	\begin{subfigure}[b]{0.49\textwidth}
		\centering
		\includegraphics[width=\textwidth]{ch3_figures-2.mps}
		\caption{\label{Gn1}}
	\end{subfigure}
	\begin{subfigure}[b]{0.49\textwidth}
		\centering
		\includegraphics[width=\textwidth]{ch3_figures-3.mps}
		\caption{\label{Gn2}}
	\end{subfigure}
	\begin{subfigure}[b]{0.49\textwidth}
		\centering
		\includegraphics[width=\textwidth]{ch3_figures-4.mps}
		\caption{\label{Gn3}}
	\end{subfigure}
	\begin{subfigure}[b]{0.49\textwidth}
		\centering
		\includegraphics[width=\textwidth]{ch3_figures-5.mps}
		\caption{\label{Gn4}}
	\end{subfigure}
	\caption{(A)~shows $G_n^1$, (B) shows $G_n^2$, (C) shows $G_n^3$, (D) shows $G_n^4$}
	\label{4graf}
\end{figure}

The splitting is based on the arrangement of the edges of $n$  in an Eulerian trail depending on whether $(a_{N-7},a_{N-6})$, $(a_{N-7},a_{N-5},a_{N-6})$, $(a_{N-7},a_{N-5},a_{N-8})$ or $(a_{N-7},a_{N-8})$ is a subtrail of it. Since there are no more possible arrangements of the edges of $n$ in an Eulerian trail, Eq. \ref{thm:Sn} holds.

We only have to calculate $S(G_n^4,B)$ since 
$$ S(G_n^1,B)\overset{\vphantom{f}}{=}S(G_n^2,B)=S(G_n^3,B)=0 $$
follows from Cor. \ref{obs2}/1 and Eq. \ref{eq:S(G,B)} applied to the following three cases:

\begin{enumerate}
\item $q_1:\overset{\vphantom{f}}{=}\big((n-1,n),(n,n),(n,1)\big)$ and $q_2:=(n-1,1)$ in $G_n^1$,

\item $q_3:\overset{\vphantom{f}}{=}\big((1,n),(n,n),(n,n-1)\big)$ and $q_4:=(1,n-1)$ in $G_n^2$,

\item  $q_5:\overset{\vphantom{f}}{=}\big((1,n),(n,n),(n,1)\big)$ and $q_6:=(1,1)$ in $G_n^3$
\end{enumerate}
\noindent where $(n-1,n)$ denotes $a_{N-9}$ and so on.

Let us notice that we get the graph $G_n^4$ from $G_{n-1}$ by adding a two-edges-long circle on the node $1$ and by replacing the loop on the node $n-1$ with a three-edges-long trail. Shifting the even-long circle does not change the sign of  $\sigma\in P(G_n^4)$ and neither the sign of $\sigma_{M(B)}$ because this circle does not include edges from $B$. So we only have to determine the sign of one permutation corresponding to an Eulerian trail in $G_n^4$ and the number of potential places for this circle. 

We claim that 
$$ S(G_n^4,B)\overset{\vphantom{f}}{=}(\m+n-1) S(G_{n-1},B) $$
since there are $\deg^-_{G_{n-1}}(1)=\m+2+(n-3)=\m+n-1$ different edges before which we can put the two-edges-long circle, moreover
for every $\pi \in P(G_{n-1}):$ $\mathrm{sgn}(\pi)=\mathrm{sgn}(\pi')$ where
\begin{equation*}\label{pi'}
\big(\pi'(i)\big)_{i=1}^{N-5}=\big(N-7,\,N-8,\,N-5,\,N-6,\,\pi(1),\,\pi(2),\,...\,,\pi(N-9)\big)
\end{equation*}
and we have to shift even-long parts to get an element of $P(G_n^4)$ from $\pi'$.

	 Thus we have
$$\label{S1} S(G_n,B)=(\m+n-1) S(G_{n-1},B)$$
	which is the required formula.
\end{proof}

\subsection{The initial step}
We solve the case of two nodes (see Fig. \ref{G_2}) by the following method. We choose an Eulerian trail and generate all Eulerian trails uniquely from the chosen one by specified steps. We determine the number of the Eulerian trails the steps give in a sign-preserving, resp. sign-reversing way. Summarizing these results we get $S(G_2,B)$.

To verify the formula in Lemma \ref{lem:S3} we combine the idea of the proof of Lemma \ref{lem:recursive} and \ref{lem:S2}. That is, we partition the set of the Eulerian trails, using this we prove that some of the summands in $S(G_3,B)$ are equal to zero, some may be obtained from the Eulerian trails of $G_2$ and some need further investigation. We classify the remaining Eulerian trails by the longest trail having a specific property, resp. by their last edge. We manage these cases applying the method we use in the proof of Lemma \ref{lem:S2}.

\begin{lem}\label{lem:S2}
	$S(G_2,B)=\big((\m+1)!\big)^2.$
\end{lem}

\begin{proof}
The first observation we have to do is that $\id\in P(G_2)$ i.e. $(a_1,a_2,...,a_{m+3})$ is an Eulerian trail in $G_2$ and we can get any Eulerian trail of $G_2$ uniquely from that with the following four steps, applying each step at most once, in the given order.

\begin{enumerate}
\item Changing the arrangement of the edges of $B$ in the first $\m$ circles.

\item Transposing the parallel edges $a_{m+2}$ and one of the $(1,2)$ edges.

\item Shifting the edge $a_{m+1}$ before one of the $(1,2)$ edges.

\item Shifting the edge $a_{m+3}$ after one of the $(1,2)$ edges.
\end{enumerate}

Let us calculate how many Eulerian trails we have. We claim that $$|P(G_2)|=\big((\m+1)!\big)^2(\m+1)$$ since we have to multiply $|\{\id\}|=1$ by $(\m!)^2$ because of step (1), and step (2), (3) and (4) implies an $(\m+1)$ factor each.

To determine $S(G_2,B)$ we have to examine whether the above steps are sign-preserving or sign-reversing, where by sign we mean the map $\sigma\mapsto\s(\sigma)\s(\sigma_{M(B)})$ as we defined it at the end of the Subsection \ref{Notations}.

It is not hard to check by Obs. \ref{obs} that the first three steps are sign-preserving. In case of step (4) it depends on the position of the chosen $(1,2)$ edge - let us denote this edge by $a$ - such that it is sign-preserving if $a_{m+1}$ precedes $a$ and sign-reversing otherwise, since $\s(\id_{M(B)})$ does not change shifting an edge which does not belong to $B$. Thus, to handle this case, we have to consider the position of $a_{m+1}$. If there are $j$ two-edges-long circles before $a_{m+1}$, then the shift of $a_{m+3}$ to the potential positions implies a $(-j+(\m-j)+1)=(\m+1-2j)$ factor. Therefore the effect of step (3) and (4) should be examined together. Since $0\leq j\leq \m$ the multiplying factor implied by these steps is
$\sum\limits_{j=0}^{\m}(\m+1-2j)=\m+1$.

Summarizing the above factors
$$S(G_2,B)=1\cdot (\m!)^2\cdot (\m+1)\cdot (\m+1)=\big((\m+1)!\big)^2$$ as we claimed.
\end{proof} 
\begin{figure}[!h]
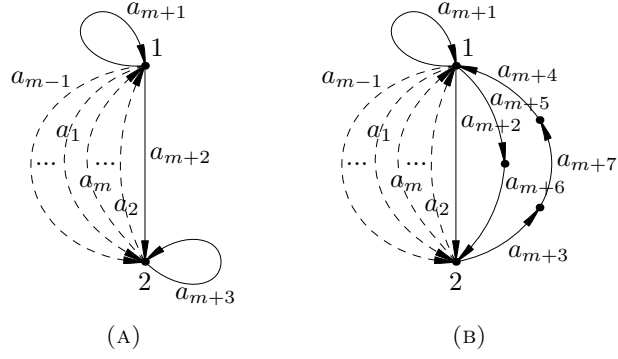

	\centering
	\begin{subfigure}[b]{0.35\textwidth}
		\centering
		\includegraphics{ch3_figures-6.mps}
		\caption{\label{G_2}}
	\end{subfigure}
	\begin{subfigure}[b]{0.35\textwidth}
		\centering
		\includegraphics{ch3_figures-7.mps}
		\caption{\label{G_3^1}}
	\end{subfigure}
	\caption{(A)~shows $G_2$ and~(B) shows $G_3^1$}
\end{figure}\label{G2,G3}

\begin{lem}\label{lem:S3}
	$S(G_3,B)=\frac23\m(\m+2)!\m!$.
\end{lem}
\begin{proof}
We split the node $3$ based on the arrangement of its edges as we did it in the proof of Lemma \ref{lem:recursive}. Thus we get the equation 
\begin{equation*}S(G_3,B)=\sum_{i=1}^{4}S(G_3^i,B)\end{equation*}
where the graphs $G_n^1$, $G_n^2$, $G_n^3$, $G_n^4$ are defined as in Fig. \ref{4graf} with the edge-enumeration inherited from $G_n$.

For $G_3^2$, $G_3^3$ and $G_3^4$ the reasoning applied in the proof of Lemma \ref{lem:recursive} works, hence
\begin{equation}\label{eq:S(G_3,B)}S(G_3,B)=S(G_3^1,B)+(\m+2)S(G_2,B)\end{equation}
holds. For $G_3^1$ see Fig. \ref{G_3^1}.

To determine $S(G_3^1,B)$ we cannot repeat the argument used in the proof of Lemma \ref{lem:recursive} since the node $n-1$ and the node $2$ are the same, so one of the parallel subtrails that we have chosen would contain an edge from $B$.

We investigate the trails from $1$ to $1$ that do not contain the node $1$ as an intermediate node. Such trails will be called \textit{$1$-trails}. Considering the in- and outdegrees of the nodes it is obvious that every Eulerian trail of $G_3^1$ contains exactly $\overline{m}+2$ $1$-trails. 
Let us denote the unique longest $1$-trail by $q_1$ and notice that $q_1$ may consist of five or four edges. 

We split the graph $G_3^1$ according to which edges are contained by $q_1$. It is straightforward to check that there are $\m+2$ cases and we denote these graphs by $G^{5}$, $G^{4}$, $G^{4\sim i}$, $i\in[\m]$ where $G^{5}$ stands for the case when $|q_1|=5$, $G^{4}$ for the case when $|q_1|=4$, $|q_1|'=0$ (i.e. there is no edge in $q_1$ that belongs to $B$) and $G^{4\sim i}$ for the case when $|q_1|=4$, $|q_1|'=1$, $a_{2i-1}\in q_1$.

Let us notice that the subsets $P(G^{5})$, $P(G^{4})$, $P(G^{4\sim 1})$,...,$P(G^{4\sim \m})$ partition $P(G_3^1)$ thus
$$ S(G_3^{1},B)=S(G^{5},B)+S(G^{4},B)+\sum_{i=1}^{\m}S(G^{4\sim i},B).$$
We get $S(G^{5},B)=0$ from Cor. \ref{obs2}/1 and Eq. \ref{eq:S(G,B)} applied to $q_1$ and $a_{m+1}$ in $G^5$.  

To calculate $S(G^{4},B)$ and $S(G^{4\sim i},B)$ we have to examine which can be the last edge in an Eulerian trail. It also determines the length of the $1$-trails. Let us introduce the notation
\begin{equation*}P^{x}(G)\overset{\mathrm{def}}{=}\{\sigma\in P(G)\,|\, \sigma(m+7)=x\},\quad G\in\{G^4, G^{4\sim i}\}. \end{equation*}

In an Eulerian trail of the graph $G^4$ the last trail from $1$ to $2$ can be $(a_{m+5},a_{m+6})$ or $a_{2\ell-1}$ for $\ell\in[\m]$. Accordingly 
$$P(G^4)=P^{m+6}(G^4)\sqcup \bigsqcup_{\ell=1}^{\m} P^{2\ell-1}(G^4).$$
If the last edge is $a_{2\ell-1}$ ($\ell\in[\m]$), then there exists a $j\in[\m]$ such that $q_2:=(a_{m+5},a_{m+6},a_{2j})$ is a $1$-trail. Let us apply Obs. \ref{obs} to the parallel subtrails $q_2$ and $a_{m+1}$, for which -- using the notation of the observation -- $h'$ is even. Then we get 
$$s\big(P^{2\ell-1}(G^4),M(B)\big)=0$$ for all $\ell\in[\m]$, where $s(P,M)$ is defined in Eq. \ref{eq:s(P,M)}.

In an Eulerian trail of the graph $G^{4\sim i}$ the last trail from $1$ to $2$ can be $(a_{m+5},a_{m+6})$, $a_{m+2}$ or $a_{2\ell-1}$ for $\ell\in[\m]$, $\ell\neq i$. Accordingly
$$P(G^{4\sim i})=P^{m+6}(G^{4\sim i})\sqcup P^{m+2}(G^{4\sim i})\sqcup \bigsqcup_{\ell\neq i, \ell=1}^{\m} P^{2\ell-1}(G^{4\sim i}).$$
If the last edge is $a_{m+6}$ or $a_{2\ell-1}$ ($\ell\in[\m], \ell\neq i$), then there exists a $j\in[\m]$ such that $q_2:=(a_{m+2},a_{2j})$ is a $1$-trail. Applying Cor. \ref{obs2}/2 to $q_2$ and $q_1=(a_{2i-1},a_{m+3},a_{m+7},a_{m+4})$ we get
$$s\big(P^{m+6}(G^{4\sim i}),M(B)\big)=s\big(P^{2\ell-1}(G^{4\sim i}),M(B)\big)=0$$ for all $\ell\in[\m], \ell\neq i$.

Thus \begin{multline*}S(G_3^{1},B)=S(G^{4},B)+\sum_{i=1}^{\m}S(G^{4\sim i},B)=\\
=s\big(P^{m+6}(G^{4}),M(B)\big)+\sum_{i=1}^{\m}s\big(P^{m+2}(G^{4\sim i}),M(B)\big).\end{multline*}

For all $i\in[\m]$, there is a bijection between $P^{m+2}(G^{4\sim i})$ and $P^{m+2}(G^{4\sim \m})$ provided by the map $\sigma\mapsto\sigma^{a_{2i-1}\leftrightarrow a_{m-1}}$ which is sign-preserving since $\s(\sigma)$ and $\s(\sigma_{M(B)})$ also changes sign by the transpose of $a_{2i-1}$ and $a_{m-1}$, hence
\begin{equation}\label{eq:S(G_3^1,B)}S(G_3^{1},B)=
s\big(P^{m+6}(G^{4}),M(B)\big)+\m\cdot s\big(P^{m+2}(G^{4\sim \m}),M(B)\big).\end{equation}

To determine the terms above we apply an argument similar to the idea that we used in the proof of Lemma \ref{lem:S2}, i.e. we take an Eulerian trail corresponding to a permutation in $P^{m+6}(G^4)$, resp. in $P^{m+2}(G^{4\sim \m})$, generate all Eularian trail uniquely from the chosen one and investigate whether the steps were sign-preserving or sign-reversing.

For the first term, the chosen Eulerian trail is $(a_1,a_2,...,a_m,a_{m+1},q_1,a_{m+5},a_{m+6})$ where $q_1=(a_{m+2},a_{m+3},a_{m+7},a_{m+4})$ and we use the following steps, applying each steps at most once, in the given order.

\begin{enumerate}
\item Changing the arrangement of the edges of $B$ in the first $\m$ circles.

\item Choosing the position of $a_{m+1}$ and $q_1$ among the $\m+2$ $1$-trails by shifting.
\end{enumerate}

We claim that $|P^{m+6}(G^4)|=\m!(\m+2)!$ since we started from one Eulerian trail and step (1) implies an $(\m!)^2$ factor and step (2) an $(\m+2)(\m+1)$ factor.

It is not hard to check by Obs. \ref{obs} that all of the steps above are sign-preserving, so using that $\s(\sigma)=-1$ and $\s(\sigma_{M(B)})=1$, where $(\sigma(i))_{i=1}^{m+7}:=(1,2,...,m+3,m+7,m+4,m+5,m+6)$ corresponding to the Eulerian trail $(a_1,a_2,...,a_m,a_{m+1},q_1,a_{m+5},a_{m+6})$, we get
$$s\big(P^{m+6}(G^{4}),M(B)\big)=(-1)\cdot (\m!)^2\cdot(\m+2)(\m+1)=-\m!(\m+2)!.$$

For $P^{m+2}(G^{4\sim \m})$, the chosen Eulerian trail is $(a_1,a_2,...,a_{m-2},q_1,q_2,q_3,a_{m+2})$, where $q_1=(a_{m-1},a_{m+3},a_{m+7},a_{m+4})$ as previously defined, $q_2:=(a_{m+5},a_{m+6},a_m)$, $q_3:=a_{m+1}$ and the steps are as follows.

\begin{enumerate}
\item Transposing $a_m$ and $a_{2j}$, $j\in[\m]$.

\item Changing the arrangement of the edges of $B$ in the first $\m-1$ circles.

\item Choosing the position of $q_1$, $q_2$ and $q_3$ among the $\m+2$ $1$-trails by shifting.
\end{enumerate}

It is easy to check that $|P^{m+2}(G^{4\sim \m})|=1\cdot \m\cdot \left((\m-1)!\right)^2\cdot (\m+2)(\m+1)\m=\m!(\m+2)!$ and applying Obs \ref{obs} it is obvious that the first two steps are sign-preserving.

It can also been seen from Obs \ref{obs} that the effect of step (3) on the sign does not depend on the number of the two-edges long circles between the investigated $1$-trails. It depends only on the relative position of the three circles. The arrangements in which $q_1<q_2<q_3$ or $q_3<q_2<q_1$ holds are sign-preserving and the rest are sign-reversing.

Thus using that $\s(\sigma)=1$ and $\s(\sigma_{M(B)})=1$, where
$\sigma\in P^{m+2}(G^{4\sim \m})$ is the permutation related to the Eulerian trail $(a_1,a_2,...,a_{m-2},q_1,q_2,q_3,a_{m+2})$, we get
$$s\big(P^{m+2}(G^{4\sim \m}),M(B)\big)=1\cdot \m!(\m+2)!\cdot \left(-\tfrac{2}6\right) =-\tfrac13\m!(\m+2)!.$$

Substituting these results in Eq. \ref{eq:S(G_3^1,B)} we get 
$$S(G_3^1,B)=-\m!(\m+2)!+\m\cdot\left(-\tfrac13\m!(\m+2)!\right)$$
from which using Eq. \ref{eq:S(G_3,B)} and Lemma \ref{lem:S2} the formula
\begin{multline*}S(G_3,B)= \left(-\m!(\m+2)!+\m\cdot\left(-\tfrac13\m!(\m+2)!\right)\right)+\\
+(\m+2)\big((\m+1)!\big)^2=
\tfrac23\m(\m+2)!\m!.\end{multline*}
occurs, as we claimed.
\end{proof}

\subsection{Proof of the theorem}

\begin{proof}[Proof of Theorem \ref{thm:Sn}]
For $n=2$ it is exactly the statement of Lemma \ref{lem:S2}. For $n\geq 3$ we prove the formula in the theorem by induction. If $n=3$ then the formula holds trivially from Lemma \ref{lem:S3}.

Now let us assume that $S(G_{n-1},B)=\frac23\m(\m+n-2)!\m!$ holds (for some $n\geq 4$). Applying the statement of Lemma \ref{lem:recursive} and the induction hypothesis we get
	$$ S(G_n,B)=(\m+n-1)\cdot\tfrac23\m(\m+n-2)!\m!=\tfrac23\m(\m+n-1)!\m!$$
for all $n\geq 4$ which is the required formula.
\end{proof}

\begin{rem}
Similar statements can be proved for variants of $G_n$:
\begin{enumerate}
	\item By shifting the loops $a_{m+1}$ and $a_{m+4n-5}$ to different nodes we can easily construct further graphs $H_n$ such that $S(H_n,B)\neq 0$ still holds. 
	\item However, it is not true that we can shift both loops without any restriction, e.g. for $a_{m+3} = (3,3)$, $a_{m+4n-5}=(5,5)$ 
		and $n\geq 6$ we have $S(H_n,B)=0$.
	\item The recursive formula given in Lemma \ref{lem:recursive} remains valid if we shift only the loop $a_{m+4n-5}$ to another node $1< i < n$.
\end{enumerate}
\end{rem}

\section{Upper bound}\label{sec:upperbound}

In this section we prove Theorem \ref{thm:upper} i.e. that $s_{4n-2}$ is a polynomial identity of $M_nE^2$ and $M_nE^3$ for all $n\geq 2$. By the next remark (noted in \cite[Sec. 4]{F}), if $s_{4n-1}$ is an identity then $s_{4n-2}$ is also an identity.

\begin{rem}\label{rem:even-odd}
If $s_k$ is a polynomial identity of $M_nE^m$ then it implies that $s_{k+1}$ is a polynomial identity too. If $k$ is even then the converse also holds by $s_k(x_1,\dots,x_k) = s_{k+1}(1,x_1,\dots,x_k)$.
\end{rem}

Therefore, by Prop. \ref{prop:equivalence} and the remark, it is enough to show the following statement.
Throughout the section $G=(V,A,s,t)$ will denote a doubly-rooted digraph and $B\subseteq A$.

\begin{thm}\label{thm:upperbound}
For all $n\geq 2 $, if $G$ has $n$ vertices and $4n-1$ edges and $|B|\leq 3$ then $T(G,B)=0$.
\end{thm}

In short, the proof of Theorem \ref{thm:upperbound} is based on homogenizing the degrees of the vertices using Lemma \ref{lem:cdeg} -- which is implicitly used in \cite{S1} -- and replacing the graph with multiple modified graphs using Lemma \ref{lem:loops}.

In the following lemmas we will assume that $G$ has $n$ vertices, $4n-1$ edges and $|B|\leq 3$.  To simplify the discussion, let us define an extended doubly-rooted digraph as
\[
G^*=(V^*,A^*,r,r)\qquad
V^*=V \cup \{r\}\qquad
A^*=A\cup \{a_0,a_{4n}\}
\]
where the new edges are defined as $a_0 = (r,s)$ and $a_{4n} = (t,r)$. Moreover, fix the enumeration $A^*=\{a_0,a_1,\dots,a_{4n}\}$.
Then $S(G,B)=S(G^*,B)$ using the enumeration inherited from $A^*$.

We may also assume without loss of generality that there is a directed Eulerian trail in $G$ from $s$ to $t$, in other words in $G^*$ from $r$ to $r$. Equivalently, $G$ is weakly connected and $\mathrm{deg}_{G^*}^-(v) = \mathrm{deg}_{G^*}^+(v)$ for all $v \in V$. Denote this common value by $\mathrm{cdeg}(v)$ that stands for \emph{corrected degree}. Explicitly,
\[\mathrm{cdeg}(v)=\mathrm{deg}_G^-(v)+\delta_{s,v}=\mathrm{deg}_G^+(v)+\delta_{t,v}\]
for all $v\in V$.

\subsection{Degree-homogenization}

For any $n \geq 3$, denote by $(IH_n)$ the induction hypothesis of Theorem \ref{thm:upperbound} i.e. for any $2\leq k<n$, doubly-rooted digraph $G'$ on $k$ vertices with at least $4k-1$ edges, and for any subset of edges $B'$ with $|B'|\leq 3$, we have $S(G',B')=0$. By Remark \ref{rem:reform}, knowing the statement for every graph with $4k-1$ edges and every graph with at least $4k-1$ edges are equivalent.

Note that by Remark \ref{rem:even-odd}, if $(IH_n)$ holds then the same statement holds if $G'$ has only $4k-2$ edges.

\begin{lem}\label{lem:cdeg}
Let $n\geq 3$ and assume the induction hypothesis $(IH_n)$. Let $u\in V$ such that either of the following holds:
\begin{enumerate}
\item $\mathrm{cdeg}(u)\leq 3$ and there is a loop on $u$,
\item $\mathrm{cdeg}(u)\leq 3$ and $u\in \{s,t\}$,
\item $\mathrm{cdeg}(u)\leq 2$,
\item $\mathrm{cdeg}(u)= 4$, $u=s=t$ and there is a loop on $u$, or
\item $\mathrm{cdeg}(u)= 4$ and there are at least two loops on $u$.
\end{enumerate}
Then $S(G,B)=0$.
\end{lem}

\begin{figure}[!h]
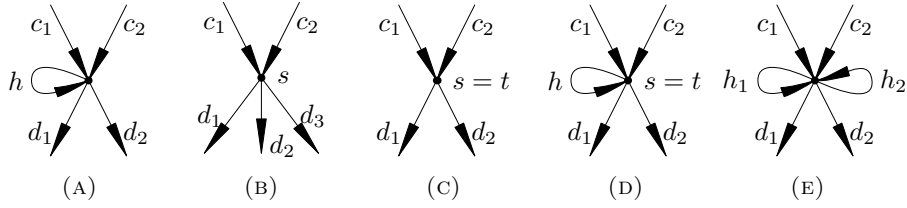

	\centering
	\begin{subfigure}[b]{0.18\textwidth}
		\centering
		\includegraphics{ch4_figures-1.mps}
		\caption{\label{fig:deg_a}}
	\end{subfigure}
	\begin{subfigure}[b]{0.18\textwidth}
		\centering
		\includegraphics{ch4_figures-2.mps}
		\caption{\label{fig:deg_b}}
	\end{subfigure}
	\begin{subfigure}[b]{0.18\textwidth}
		\centering
		\includegraphics{ch4_figures-3.mps}
		\caption{\label{fig:deg_c}}
	\end{subfigure}
	\begin{subfigure}[b]{0.18\textwidth}
		\centering
		\includegraphics{ch4_figures-4.mps}
		\caption{\label{fig:deg_d}}
	\end{subfigure}
	\begin{subfigure}[b]{0.18\textwidth}
		\centering
		\includegraphics{ch4_figures-5.mps}
		\caption{\label{fig:deg_e}}
	\end{subfigure}
	\caption{Extremal cases of Lemma \ref{lem:cdeg}}
\end{figure}

In short, we will show that $S(G,B)$ may be expressed as a sum of $S(G_\ell,B_\ell)$ where $G_\ell$ is obtained by deleting $u$ from $V(G)$ and replacing some edges in $A(G)$ so that the number of edges decreases by at most five for all $\ell$. Then $G_\ell$ will have $n-1$ vertices and at least $4(n-1)-2$ edges, hence $(IH_n)$ implies that $S(G_\ell,B_\ell)=0$. 

\begin{proof}
We will denote the loops on $u$ in $A(G)$ by $h_1, h_2, \dots$ or $h$, the incoming edges by $c_1, c_2, \dots$ and the outgoing edges by $d_1, d_2, \dots$. Given a trail $q$ (resp. trails $q_1,q_2$) in $G$ we denote by $P_q(G)\subseteq P(G)$ (resp. $P_{q_1,q_2}(G)$) the set of permutations of the edges that form an Eulerian trail and contain $q$ (resp. $q_1, q_2$) as a subtrail. As in Section \ref{Notations}, whenever we say that a map is "sign-preserving" or "sign-reversing" between two subsets of permutations, by sign we mean the map $\sigma \mapsto \mathrm{sgn}(\sigma)\mathrm{sgn}(\sigma_{M})$ where $M$ will be given by the context.

\textbf{(1)}: Assume that $u\notin \{s,t\}$, $\mathrm{cdeg}(u) = 3$ and there is a single loop $h$ on $u$, see Fig. \ref{fig:deg_a}. 
Then any Eulerian trail in $G$ has subtrails of the form $(c_i,h,d_j)$, $(c_{3-i}, d_{3-j}, f)$ for some edge $f\in A^*$ and $i,j=1,2$, therefore
\begin{equation}
\label{eq:cdeg}
P(G) = \bigsqcup P_{(c_i,h,d_j),(c_{3-i}, d_{3-j}, f)}(G).
\end{equation}
Define the graphs $G_{i,j,f}$ on $V\backslash \{u\}$ with edge-set induced from $G$ but replacing the two subtrails of length three given above by single edges $(c_i^-,d_{j}^+)$ and $(c_{3-i}^-,d_{3-j}^+)$, except in the case $f=c_i$ in when we replace the whole trail $(c_{3-i}, d_{3-i},c_i,h,d_j)$ with a single edge $(c_{3-i}^-,d_j^+)$. 

Moreover, define the subset $B_{i,j,f}\subseteq A(G_{i,j,f})$ as $B$ without $c_1,c_2,d_1,d_2,h,f$ but with those new edges for which the original subtrail corresponding to the new edge contains an odd number of edges from $B$. Then $G_{i,j,f}$ has $n-1$ vertices, $4n-5=4(n-1)-1$ edges and $|B_{i,j,f}|\leq 3$, hence $S(G_{i,j,f},B_{i,j,f})=0$ for all $i,j,f$ by $(IH_n)$. On the other hand, we may choose the ordering of the edges of each $G_{i,j,f}$ such that there is a sign-preserving bijection between $P(G_{i,j,f})$ and $P_{(c_i,h,d_j),(c_{3-i}, d_{3-j}, f)}(G)$ by Obs. \ref{obs}, hence $S(G,B)=0$ as we claimed.

Further subcases of $(1)$ -- i.e. when $\mathrm{deg}(u)<3$ or when there are multiple loops on $u$ -- may be verified analogously, still under the assumption $u\notin\{s,t\}$. The case when $u\in\{s,t\}$ is handled in the following paragraphs.

\textbf{(2)}, \textbf{(3)} and \textbf{(4)}: Assume that $u=s\neq t$, $\mathrm{cdeg}(u) = 3$ and there are no loops on $u$, see Fig. \ref{fig:deg_b}. We may apply the same argument as before with a minor twist. 
Indeed, any Eulerian trail in $G$ starts with some $d_j$ and has subtrails of the form $(c_i, d_{j+1}, f_1)$ and $(c_{3-i}, d_{j+2}, f_2)$ for $i=1,2$, $j=1,2,3$ and some edges $f_1, f_2\in A^*$ (using the convention that $d_4 = d_1$ and $d_5 = d_2$). The edges $f_1,f_2\in A^*$ exist by $u\neq t$. 

Define the graphs $G_{i,j,f_1,f_2}$ on $V\backslash \{u\}$ by replacing $(c_{i}, d_{j+1}, f_1)$ and  $(c_{3-i}, d_{j+2}, f_2)$ by single edges, and deleting $d_j$ (and setting $d_j^+$ as the source node), except if $f_1$ or $f_2$ is identical to $c_1$ or $c_2$ in which case we replace the whole continuous (odd-length) trail by a single edge. Similarly, define $B_{i,j,f_1,f_2}$ as in the previous case. Then every $G_{i,j,f_1,f_2}$ has $n-1$ vertices and $4n-1-5=4(n-1)-2$ edges, so $S(G_{i,j,f_1,f_2},B_{i,j,f_1,f_2})=0$ by $(IH_n)$. On the other hand, the analogue of Eq. \ref{eq:cdeg} holds and we may choose the ordering of the edges of $G_{i,j,f_1,f_2}$ such that there is a sign-preserving bijection between $P(G_{i,j,f_1,f_2})$ and the subsets of permutations appearing in the equation, hence $S(G,B)=0$. 
The analogous argument proves the case of $u=t\neq s$. 

Similarly, if $u=s=t$ and $\mathrm{cdeg}(u)=3$, in particular $u$ has two incoming and two outgoing edges (see Fig. \ref{fig:deg_c}), then any Eulerian trails starts with some $d_j$, ends with some $c_i$ and has a subtrail of the form $(c_{3-i},d_{3-i},f)$ so the same argument may be repeated. Further subcases of $(2)$ -- i.e. when $\mathrm{cdeg}(u)<3$ or there are more loops on $u$ --, statement (3) and (4) (see Fig. \ref{fig:deg_d}) can be verified using the same steps.

\textbf{(5)}: Assume that $u\notin\{s,t\}$, $\mathrm{cdeg}(u) = 4$ and there are exactly two loops $h_1,h_2$ on $u$, see Fig. \ref{fig:deg_e}. Denote the subtrail $(c_i, h_1, d_j)$ by $q_{i,j}$. Then
\begin{equation}
\label{eq:homogen3}
P(G)=P_{(h_1,h_2)}(G) \sqcup P_{(h_2,h_1)}(G) \sqcup \bigsqcup_{i,j=1,2} P_{q_{i,j}}(G).
\end{equation}
To deal with $P_{q_{i,j}}(G)$, define the graphs $G_{i,j}$ on $V\backslash\{u\}$ by replacing $q_{i,j}=(c_i,h_1,d_j)$ and $(c_{3-i},h_2,d_{3-j})$ by single edges. Define $B_{i,j}$ as $B$ without $h_1,h_2,c_1,c_2,d_1,d_2$ and with the new edges such that their corresponding old subtrails contain an odd number of edges from $B$. These graphs have $n-1$ vertices and $4n-5=4(n-1)-1$ edges, hence $S(G_{i,j},B_{i,j})=0$ for all $i,j$ by $(IH_n)$. On the other hand, we may choose the ordering of the edges of each $G_{i,j}$ such that there is a sign-preserving bijection between $P(G_{i,j})$ and $P_{q_{i,j}}(G)$.  

Note that if $\{h_1,h_2\}\subseteq B$ (resp. $\{h_1,h_2\}\nsubseteq B$) then the map $\sigma \mapsto \sigma^{h_1\leftrightarrow h_2}$ transposing the indices of $h_1$ and $h_2$ gives a sign-preserving (resp. sign-reversing) bijection between $P_{(h_1,h_2)}(G)$ and $P_{(h_2,h_1)}(G)$. If $\{h_1,h_2\}\nsubseteq B$ then the summands in $S(G,B)$ corresponding to $P_{(h_1,h_2)}(G)$ and $P_{(h_2,h_1)}(G)$ cancel out each other.
Therefore in this case $S(G,B)=0$ by Eq. \ref{eq:homogen3}.

If $\{h_1,h_2\}\subseteq B$ then from Eq. \ref{eq:homogen3} and the last two paragraphs we get that
\begin{equation}\label{eq:double_loop}
S(G,B) = 2\cdot s\big(P_{(h_1,h_2)}(G),M(B)\big).
\end{equation}
For $i=1,2$ define the graph $G_i$ from $G$ on the same vertex set $V$ by replacing $(c_i,h_1,h_2)$ by a single edge, and define $B_i$ from $B$ without $c_i,h_1,h_2$ and with the corresponding new edge if and only if $c_i\in B$. 
Then $G_i$ has $n$ vertices, $4n-3$ edges but $|B|\leq 1$. Therefore, by $4n-3\geq 2n$ and the Amitsur--Levitzki theorem (see \cite{AL}), we get $S(G_i,B_i)=0$ for $i=1,2$. On the other hand, there is a sign-preserving bijection between $P_{(h_1,h_2)}(G)$ and $P(G_1) \sqcup P(G_2)$ with an appropriate arrangement chosen on $G_1$ and $G_2$, hence by Eq. \ref{eq:double_loop}, $S(G,B)=2(S(G_1,B_1) + S(G_2,B_2))=0$. 

Further cases of $(5)$, i.e. if $u\in\{s,t\}$ or there are more loops on $u$ may be verified using the same steps.
\end{proof}

\begin{rem}
Note that if we assume $|A(G)|=4n-2$ instead of $4n-1$ then we cannot prove the second part of the lemma by the same inductive argument, as we had to drop 5 edges in the construction of $G_{i,j,f_1,f_2}$.
\end{rem}

\subsection{Swan's lemma}

For the next lemma we need some preliminary definitions. 
For given $a,c,d\in A$ let us define the doubly-rooted digraphs $G_{a,c}$ and $G^{a,d}$ as digraphs on the same vertex set and roots as $G$ but with edge sets
\begin{align*}
A(G_{a,c})=
&A\backslash\{a,c\}\cup\{a^+ \overset{\overline{a}}{\longrightarrow} a^+,\, 
c^- \overset{c_{a}}{\longrightarrow} a^+\},\\
A(G^{a,d})=
&A\backslash\{a,d\}\cup\{a^+ \overset{\overline{a}}{\longrightarrow} a^+,\, 
a^- \overset{d^{a}}{\longrightarrow} d^+\}.
\end{align*}
We will only apply the constructions if $a^-=c^+$ and if $a^+ = d^-$.

The arrangement of the edges of $G_{a,c}$ (resp. $G^{a,d}$) is derived from the arrangement of the edges of $G$ by replacing $a$ by $\overline{a}$ and $c$ by $c_a$ (resp. $d$ by $d^a$).
Similarly, denote by $B_{a,c}\subseteq A(G_{a,c})$ (resp. $B^{a,d}\subseteq A(G^{a,d})$) the set obtained from $B$ by replacing $a$ by $\overline{a}$ and $c$ by $c_a$ (resp. $d$ by $d^a$) if any of them are contained by $B$. In particular $M(B_{a,c}) = M(B) = M(B^{a,d})$ using the notations of Section \ref{sec:reformulation}.

Although $a_0,a_{4n}\in A^*\backslash A$ are not edges of $G$, we also define $G_{a,a_{0}}$ (resp. $G^{a,a_{4n}}$) as the subdigraph of $(G^*)_{a,a_0}$ (resp. $(G^*)^{a,a_{4n}}$) spanned by all the vertices except $r$. Explicitly, the edge set of $G_{a,a_{0}}$ (resp. $G^{a,a_{4n}}$) is
$(A\backslash\{a\})\cup\{\overline{a}\}$
with pair of roots $(a^+,t)$ (resp. $(s,a^-)$) in place of $(s,t)$.

\begin{lem}\label{lem:loops}
For any $a\in A$ that is not a loop,
\[
S(G,B)=
\sum_{c\in A^*:\,c^+=a^-} S(G_{a,c},B_{a,c}) - 
\sum_{d\in A^*:\,d^-=a^+} S(G^{a,d},B^{a,d})
\]
with the enumerations given above.
\end{lem}

\begin{figure}[h!]
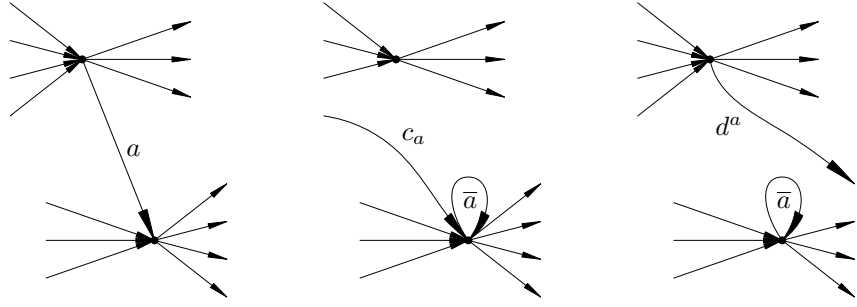

	\begin{tabular}{ll}
	\includegraphics{ch4_figures-11.mps}
	\hspace{1cm}
	\includegraphics{ch4_figures-12.mps}
	\hspace{1cm}
	\includegraphics{ch4_figures-13.mps}
	\end{tabular}
	\caption{Graph modifications in Lemma \ref{lem:loops}}\label{fig:loops}
\end{figure}

Note that Fig. \ref{fig:loops} may be misleading since $c$ (resp. $d$) may also be a loop on $a^-$ (resp. $a^+$), but the statement still applies.

\begin{proof}
To avoid separate arguments for the exceptional cases of $\{a^+,a^-\}\cap\{s,t\}\neq \emptyset$, we will work with the Eulerian trails of $G^*$ and its suitable modifications (see above), since $G$ and $G^*$ have the same Eulerian trails (as doubly-rooted digraphs), and this bijection of trails preserves the associated signs.

Recall the definition of $M(B)$, $P(G)$ and $s(P,M)$ from Eq. \ref{eq:M(B)}, \ref{eq:P(G)} and \ref{eq:s(P,M)}. By construction we may identify $P(G)$ with $P(G^*)$. Moreover, using the enumerations given before the lemma, we have
$M(B_{a,c})=M(B)=M(B^{a,d})$, hence
$\sigma_{M(B_{a,c})}=\sigma_{M(B)}=\sigma_{M(B^{a,d})}$
for any $c,d\in A^*$ and $\sigma \in P(G^*)$. Therefore we consider $M(B)$ fixed and simply write $s(P(G))$ instead of $s(P(G),M(B))$.

We have to show that
\[
s(P(G))
= \sum_c s(P(G_{a,c}))
- \sum_d s(P(G^{a,d}))
\]
where the sums run on all $c,d\in A^*$ such that $c^+=a^-$ and $d^-=a^+$. The proof is based on building partial bijections between sets of permutations. 

For all $c,d \in A^*$ such that $c^+=a^-$ and $d^-=a^+$ denote by $P_{c,d}$ the set of permutations of $\{0,1,\dots,4n\}$ such that 
$(a_{\sigma(0)},\dots,a_{\sigma(4n)})$
is an Eulerian trail of $(G^*)_{a,c}$ and $(c_a,d)$ is a subtrail of it. Similarly, denote by $P_{c,a}$ the set of permutations where $(c_a,\overline{a})$ is a subtrail of the Eulerian trail $(a_{\sigma(0)},\dots,a_{\sigma(4n)})$ of $(G^*)_{a,c}$.

We show that it is enough to verify the following equations:
\begin{align}
\centering
P(G)		& = \bigsqcup_{c}P_{c,a}
\label{eq:p1}\\
P(G_{a,c})	& = P_{c,a} \sqcup \bigsqcup_{d}P_{c,d} \qquad (\forall c\in A^*:\,c^+=a^-)
\label{eq:p2}\\
P(G^{a,d})	& = \bigsqcup_{c}P_{c,d}		\hspace{1.7cm} (\forall d\in A^*:\,d^-=a^+)
\label{eq:p3}
\end{align}
where the unions run on all $c,d\in A^*$ such that $c^+=a^-$ and $d^-=a^+$. Indeed, since $P\mapsto s(P)$ is additive under disjoint union, we have
\begin{gather*}
s(P(G))=\sum_{c}s(P_{c,a}) = \\
= \sum_{c} \big(s(P_{c,a}) + \sum_d s(P_{c,d})\big)
- \sum_{d} \bigg(\sum_{c}s(P_{c,d})\bigg) \\
= \sum_c s(P(G_{a,c})) - \sum_d s(P(G^{a,d}))
\end{gather*}
and that was the claim.

The verifications of Eq. \ref{eq:p1}, \ref{eq:p2} and \ref{eq:p3} are done by case-checking. In the case of Eq. \ref{eq:p1} (resp. Eq. \ref{eq:p3}) one has to check which edge precedes $a$ in an Eulerian trail of $G$ (resp. $d^a$ in an Eulerian trail of $G^{a,d}$). In Eq. \ref{eq:p2} one has to check which edge follows $c_a$ in $G_{a,c}$: it is either $\overline{a}$ or one of the $d$'s. Since $a\in A$, it cannot be the first or the last edge of an Eulerian trail of $G^*$. The claim follows.
\end{proof}

The next claim is a combination of Lemma \ref{lem:cdeg} and \ref{lem:loops} that is used multiple times in the proof of Theorem \ref{thm:upperbound}.

We will need the notion of the \emph{opposite} of a doubly-rooted digraph $G$, denoted by $G^{\mathrm{op}}$. The opposite is constructed from $G$ by reversing the direction of every edge, and transposing $s$ and $t$ (while keeping $B$ fixed). Note that $S(G,B) = S(G^\mathrm{op},B)$ using $|A|=4n-1$.

\begin{cor}\label{cor:path}
Let $n\geq 3$ and assume $(IH_n)$. Let $f_1,f_2,\dots, f_j$ be an undirected path in the digraph $G$ for some $j\geq 1$ (with no repeated vertices), starting at $s$ or $t$ and ending at some $v_0\in V$. Assume also that $f_j^+ = v_0$, moreover for every vertex $v\neq v_0$ touched by the path, we have $\mathrm{cdeg}(v) = 4$ and either $v\in\{s,t\}$ or there is exactly one loop on $v$. Then 
\[S(G,B) = \sum_{d\in A^*:\,d^-=f_j^+} - S(G^{f_j,d},B^{f_j,d}).\]
In particular, if $\mathrm{cdeg}(v_0)=4$ and there is a loop $h$ on $v_0$ then
\[S(G,B) = - S(G^{f_j,h},B^{f_j,h}).\]
\end{cor}

\begin{figure}[h!]
	\includegraphics{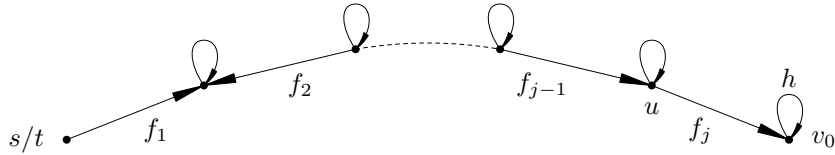}
	\caption{A path as in Corollary \ref{cor:path}}
\end{figure}

\begin{proof}
For the first claim, by Lemma \ref{lem:loops}, we have to prove that $S(G_{f_j,c},B_{f_j,c})=0$ for all $c$ such that $c^+ = f_j^- =: u$. We may apply induction on $j$. 
Note that $\mathrm{cdeg}_{G_{f_j,c}}(u) = 3$. If $u \in \{s,t\}$ (in particular if $j=1$) then $S(G_{f_j,c},B_{f_j,c})=0$ by Lemma \ref{lem:cdeg}/2. Assume that $u\notin\{s,t\}$, in particular $j>1$. By the assumption, $u$ has a loop attached to it. If $c$ is not this loop itself, then $u$ has a loop in $G_{f_j,c}$ too, hence $S(G_{f_j,c},B_{f_j,c})=0$ by Lemma \ref{lem:cdeg}/1. 

If $c$ is the loop on $u$, then we may apply the induction on $j$ for the path $f_1,\dots ,f_{j-1}$ in $G_{f_j,c}$. The only problem that may arise is that instead of  $f_{j-1}^+ = u$ we may have $f_{j-1}^- = u$. In that case we can replace $G_{f_j,c}$ by its opposite graph $(G_{f_j,c})^\mathrm{op}$, and apply the induction on that graph. By the induction hypothesis
\[S(G_{f_j,c},B_{f_j,c}) = \sum_{d\in A^*(G_{f_j,c}):\,d^-=f_{j-1}^+} - S\big((G_{f_j,c})^{f_{j-1},d},(B_{f_j,c})^{f_j,d}\big).\]
However, 
$\mathrm{cdeg}_{(G_{f_j,c})^{f_{j-1},d}}(u) = \mathrm{cdeg}_{G_{f_j,c}}(u) = 3$
for any $d$, and $\overline{f_{j-1}}$ is a loop on $u$, hence every summand is zero by Lemma \ref{lem:cdeg}/1. Therefore the first claim holds by induction.

For the "in particular" part, note that if $d\neq h$ then in $G^{f_j,d}$ there are at least two loops on $v_0$. On the other hand $\mathrm{cdeg}(v_0)=4$ still holds in $G^{f_j,d}$, hence $S(G^{f_j,d},B^{f_j,d})=0$ by Lemma \ref{lem:cdeg}/5. 
\end{proof}

\subsection{The non-2-connected case}

First we will prove Theorem \ref{thm:upperbound} for a special case. Let us define \emph{condition} $(D)$ as follows: $\mathrm{cdeg}(v)=4$ for all $v\in V$, moreover one of the following holds:
\begin{itemize}
\item $s\neq t$ and there is exactly one loop on each node, or
\item $s=t$ and there is exactly one loop on each element of $V\backslash \{s\}$.
\end{itemize}
Note that condition $(D)$ is invariant under taking the opposite of $G$.

\begin{lem}\label{lem:conn}
Let $n\geq 3$ and assume that $(IH_n)$ and $(D)$ hold for $G$. If there is a $v\in V$ such that $G\backslash\{v\}$ is not weakly connected and there is a weakly connected component of $G\backslash\{v\}$ that contains neither $s$ nor $t$, then $S(G,B)=0$.
\end{lem}

\begin{figure}[!h]
	\centering
	\begin{subfigure}[b]{0.19\textwidth}
		\centering
		\includegraphics{ch4_figures-31.mps}
		\caption{\label{fig:comp_a}}
	\end{subfigure}
	\begin{subfigure}[b]{0.19\textwidth}
		\centering
		\includegraphics{ch4_figures-32.mps}
		\caption{\label{fig:comp_b}}
	\end{subfigure}
	\begin{subfigure}[b]{0.19\textwidth}
		\centering
		\includegraphics{ch4_figures-33.mps}
		\caption{\label{fig:comp_c}}
	\end{subfigure}
		\begin{subfigure}[b]{0.19\textwidth}
		\centering
		\includegraphics{ch4_figures-34.mps}
		\caption{\label{fig:comp_d}}
	\end{subfigure}
	\begin{subfigure}[b]{0.19\textwidth}
		\centering
		\includegraphics{ch4_figures-35.mps}
		\caption{\label{fig:comp_e}}
	\end{subfigure}
	\caption{Cases of Lemma \ref{lem:conn}}
\end{figure}

\begin{proof}
First assume that $v=s=t$ and there is no loop on $v$, see Fig. \ref{fig:comp_a}. By $(D)$, there are three outgoing and three incoming edges attached to $v$. As $G\backslash \{v\}$ is not weakly connected (but $G$ was assumed to be weakly connected in the beginning of the section, otherwise $S(G,B)=0$ follows automatically), there is a weakly connected component such that there is only one edge from $v$ to the component and one in the reverse direction. Denote by $G_1$ the subgraph induced by this component together with $v$.

We claim that $S(G_1,B\cap A(G_1))$ divides $S(G,B)$, where both roots of $G_1$ are defined as $v$. Indeed, first note that each Eulerian trail of $G$ contains a subtrail that is an Eulerian trail of $G_1$, by $\mathrm{deg}^{\pm}_{G_1}(v)=1$. Classify the Eulerian trails of $G$ based on the position of the edges in $A(G)\backslash A(G_1)$. In other words, consider the classes of the equivalence relation defined as follows: two Eulerian trails of $G$ are equivalent if and only if they only differ by a permutation of $A(G_1)$. For any fixed class there is a bijection between the Eulerian trails of $G_1$ and the elements of the class. This bijection is either sign-preserving or sign-reserving on the whole class (where sign means $\sigma\mapsto \mathrm{sgn}(\sigma)\mathrm{sgn}(\sigma_M)$ as before), hence $S(G_1,B\cap A(G_1))$ divides the summand of $S(G,B)$ corresponding to the class. As this holds for each class, the claimed divisibility holds as well.

Therefore, for the case of $v=s=t$ and no loop on $v$, it is enough to prove that  $S(G_1,B\cap A(G_1))=0$. By $(D)$, $\mathrm{cdeg}(u)=4$ for any $u\in V(G_1)\backslash \{v\}$, and $\mathrm{deg}^{\pm}_{G_1}(v)=1$ by definition. So we may apply $(IH_n)$ on $G_1$ as it is isomorphic to the extended digraph $H^*$ of a digraph $H$ satisfying the assumptions of the proposition, but $|A(H)|<|A(G)|$.

If $v=s=t$ but there is a loop on $v$ (see Fig. \ref{fig:comp_b}), then the same argument applies (if we define $G_1$ without the loop on $v$), or alternatively, we may apply Lemma \ref{lem:cdeg}/4. 

Assume that $v=s\neq t$ (the case of $v=t\neq s$ is analogous). By $(D)$, the edges attached to $v$ in $G$ consist of a loop, three outgoing and two incoming edges. 
There are two cases: if there is a component that does not contain $t$ and is connected with one incoming and one outgoing edge to $v$ (see Fig. \ref{fig:comp_c}) then we may apply the previous argument. In the other case there is only one outgoing edge $a$ from $v$ to the component containing $t$ (and none backwards, see Fig. \ref{fig:comp_d}). Then we may apply the same divisibility argument on the component of $t$ as before, where the source is chosen to be $a^+$. The assumptions of the induction hypothesis hold for the component of $t$ by condition $(D)$, hence this case is solved by $(IH_n)$.

Now assume that $v\notin \{s,t\}$. Then by $(D)$, there are three outgoing and three incoming edges attached to $v$ beyond the loop. To apply the previous argument in this case, we have to use the assumption that one of the components contains none of $s$ and $t$. If there is only one edge from this component to $v$ and one backwards, then we are in the same situation as in the very first case, hence we are done by the induction hypothesis $(IH_n)$.

The last case is where $G\backslash \{v\}$ has two components, one with $s$ and $t$ that is connected to $v$ by one incoming and one outgoing edge, and the other component with twice that many connecting edges, see Fig. \ref{fig:comp_e}. In that case denote by $G_2$ the subgraph induced by $v$ and the component that does not contain $s$ and $t$ (including the loop on $v$). 
Again, note that each Eulerian trail of $G$ contains a subtrail that is an Eulerian trail of $G_2$ (including the loop on $v$) from $v$ to $v$, by degree considerations. Set the source and target of $G_2$ as $v$, then we may derive that $S(G_2,B\cap A(G_2))$ divides $S(G,B)$ the same way as in the first case. On the other hand, $\mathrm{cdeg}_{G_2}(u)=4$ for all $u\in V(G_2)$ by $(D)$ and $|V(G_2)|<|V(G)|$, so we may apply $(IH_n)$ on $G_2$.

We covered every case allowed by $(D)$ and the assumptions of the lemma, hence the claim follows.
\end{proof}
\vspace{0cm}
\subsection{Proof of the theorem}

Although the idea of the proof of Theorem \ref{thm:upperbound} is based on multiple reduction steps to graphs with simpler structure (in a sense discussed below), for a logically straightforward argument we will move from the special case to the more general one.

The proof applies induction on the minimal "lazy distance" $p$ among the elements of $B$ (see below).

\begin{prop}\label{prop:case-D}
Let $n\geq 3$ and assume that $(IH_n)$ and $(D)$ hold for $G$. Then $S(G,B)=0$.
\end{prop}

\begin{proof}
We may assume that $|B|>1$, otherwise we may apply the Amitsur--Levitzki theorem.
Define the lazy distance $p$ as the least natural number such that there exists an undirected walk $e_0, e_1, \dots, e_p, e_{p+1}$ in $G$ where $e_0 = b_1$, $e_{p+1} = b_2$ for some distinct elements of $b_1,b_2 \in B$ and either every odd or every even numbered edge is a loop. If there is no such undirected walk, we define $p$ as infinity.

It may happen that $p$ is not finite, i.e. such a walk does not exist. This may arise only if every walk connecting $b_1$ and $b_2$ touches a vertex without a loop (and that vertex is not the first or last node of the walk). By condition $(D)$, this implies $s=t$ and the loopless vertex is $s$. Moreover, then $G\backslash \{s\}$ is not weakly connected, hence $S(G,B)=0$ by Lemma \ref{lem:conn}. Therefore we may assume that $p$ is finite.

To prove the case of $p=0$ it is enough to consider the case when $b_2$ is a loop on a vertex $v_0$ and either $b_1^-=v_0$ or $b_1^+=v_0$ (if needed, with reversing the roles of $b_1$ and $b_2$). Since replacing $G$ with its opposite $G^{\mathrm{op}}$ does not affect neither our assumptions, nor the required claim (i.e. the value of $p$, condition $(D)$ and $S(G,B)=0$), it is enough to consider the case of $b_1^+=v_0$. By Lemma \ref{lem:cdeg}/5, we may assume that $b_2$ is the only loop on $v_0$. 

We claim that there is an undirected path $f_1, \dots, f_j$ as in Cor. \ref{cor:path} with $f_j = b_1$ and $h = b_2$. Indeed, the only problem can be that any undirected path connecting $b_1^-$ with $s$ or $t$ touches $v_0$. In that case either $v_0=s=t$ (hence $S(G,B)=0$ by Lemma \ref{lem:cdeg}/4) or $G\backslash \{v_0 \}$ is not weakly connected, and the component of $b_1^-$ does not contain $s$ and $t$ (hence $S(G,B)=0$ by Lemma \ref{lem:conn}). Therefore we may take a path as in Cor. \ref{cor:path}, every other assumption of the corollary (including $\mathrm{cdeg}(v_0)=4$ and that there is a loop on $v_0$) follows from $(D)$. Therefore
\begin{equation}\label{eq:graph-iso}
S(G,B) = -S(G^{b_1,b_2},B^{b_1,b_2}).
\end{equation}
On the other hand, there is an isomorphism of doubly-rooted digraphs between $G$ and $G^{b_1,b_2}$ via $b_1 \mapsto (b_2)^{b_1}$ and $b_2 \mapsto \overline{b_1}$ but leaving all other edges unchanged. 
This map gives a bijection on the Eulerian trails of the two graphs, that reverses both the sign of $\sigma$ and of $\sigma_{M(B)}$ in each case. Therefore we also get 
\begin{equation*}\label{eq:graph-iso-two}
S(G,B) = S(G^{b_1,b_2},B^{b_1,b_2}),
\end{equation*}
hence both sides are zero by Eq. \ref{eq:graph-iso}. This completes the case when $p=0$.

Now assume that $p\geq 1$ and take the shortest walk $e_0, e_1, \dots, e_p, e_{p+1}$ as in the definition of $p$. As the first subcase, assume that $e_0 = b_1$ is not a loop. Similarly to the case of $p=0$, we may assume that $e_1$ is the loop on $e_0^+$ (and not $e_0^-$) by replacing $G$ by its opposite, if needed. 
Then we may find an undirected path as in Cor. \ref{cor:path} such that $f_j = e_0 = b_1$ and $h = e_1$. Indeed, such a path exists by Lemma \ref{lem:conn} and Lemma \ref{lem:cdeg}/4 by the same argument we used in the previous paragraph. Therefore $S(G,B) = -S(G^{e_0,e_1},B^{e_0,e_1})$ by the corollary. 
Although the map $e_0 \mapsto (e_1)^{e_0}$ and $e_1 \mapsto \overline{e_0}$ is still a digraph-isomorphism between $G^{e_0,e_1}$ and $G$, but it does not map $B$ into $B^{e_0,e_1}$. In fact, the lazy distance between the elements of $B^{e_0,e_1}$ is at most $p-1$, using the walk $\overline{e_0}, e_2, \dots, e_{p+1}$. Therefore $S(G,B)=0$ by the induction on $p$, using that condition $(D)$ still holds for $G^{e_0,e_1}\cong G$.

As the other subcase of the inductive step on $p$, assume that $e_0 = b_1$ is a loop on a node $v_0$. Repeat the previous argument with reversing the roles of $e_1$ and $e_0$. By replacing $G$ by its opposite (if needed), we may assume that $e_1^+=e_0^-$. Then we may bound the lazy distance between the elements of $B^{e_0,e_1}$ by $p-1$ using the walk $(e_0)^{e_1}, e_2, \dots, e_{p+1}$. Still, $S(G,B)=0$ by the induction on $p$.
\end{proof}

Finally we prove the general case of the theorem, using multiple induction on $n$ (the number of vertices), $\ell$ (the number of vertices that have no loop attached to them) and $j$ (the number of edges in the undirected path defined in Cor. \ref{cor:path}).

\begin{proof}[Proof of Theorem \ref{thm:upperbound}]
We prove by induction on $n$. By \cite[Prop. 8]{F} the standard identity of degree 
$2\left(\left\lfloor\frac{n^2+1}{2}\right\rfloor+\left\lfloor\frac{m}{2}\right\rfloor\right)$ 
(in particular, for $n=2$ and $m=2,3$ the identity of degree 6) holds in $M_nE^m$. By Remark \ref{rem:reform}/2 this means the statement holds for $n=2$.

Now assume that $n\geq 3$ and the induction hypothesis $(IH_n)$ holds. We prove the theorem by induction on the number of vertices $\ell$ that have no loop attached to them. 

First if $\ell = 0$ then we may assume that $\mathrm{cdeg}(v)\geq 4$ for all $v\in V$ by Lemma \ref{lem:cdeg}/1. In fact, then $\mathrm{cdeg}(v) = 4$ for all $v$ by 
\begin{equation}\label{eq:4n}
\sum_{v \in V}\mathrm{cdeg}(v) = |A(G)| + 1 = 4n.
\end{equation}
Moreover, if there is more than one loop on any $v \in V$ then $S(G,B)=0$ by Lemma \ref{lem:cdeg}/5. Therefore we may assume condition $(D)$, hence the statement holds by \mbox{Prop. \ref{prop:case-D}.}

Let $\ell = 1$ and denote by $v_0$ the node that has no loop. By Lemma \ref{lem:cdeg}/1, $\mathrm{cdeg}(v)\geq 4$ for all $v\neq v_0$. Hence, by Eq. \ref{eq:4n} and Lemma \ref{lem:cdeg}/3, we may assume that $\mathrm{cdeg}(v_0) \in \{3,4\}$. If $v_0=s=t$ and $\mathrm{cdeg}(v_0)=3$ then $S(G,B)=0$ by Lemma \ref{lem:cdeg}/2. If $v_0=s=t$ but $\mathrm{cdeg}(v_0)=4$ then condition $(D)$ holds by Eq. \ref{eq:4n} and $\ell=1$, hence $S(G,B)=0$ by Prop. \ref{prop:case-D}. 

Assume that $\mathrm{cdeg}(v_0)=4$ and $v_0\neq s$ (the case of $v_0\neq t$ is analogous). Choose an undirected path $f_1,\dots, f_j$ without repeated vertices, starting at $s$ and ending at $v_0$. We may assume that $f_j^+ = v_0$, by replacing $G$ by its opposite if needed. Then we may apply Cor. \ref{cor:path}. 
It is enough to prove that $S(G^{f_j,d},B^{f_j,d})=0$ for any $d\in A^*$ such that $d^- = f_j^+$. But that is true, since $G^{f_j,d}$ has a loop on all of its vertices, so we may apply the induction on $\ell$.

As other subcase of $\ell = 1$, assume that $\mathrm{cdeg}(v_0) = 3$ and $v\neq s$ (the case of $v_0\neq t$ is analogous). Then there is a vertex $v_1$ such that $\mathrm{cdeg}(v_1) = 5$, and  $\mathrm{cdeg}(v) = 4$ for every $v\in V\backslash \{v_0,v_1\}$. Let $f_1, \dots, f_j$ be a shortest undirected path from $v_1$ to $v_0$ (without repeated vertices). By replacing $G$ by its opposite (if needed), we may assume that $f_j^+ = v_0$.

We will prove by induction on $j$.
Apply Lemma \ref{lem:loops} on $f_j$. 
If $d \in A^*$ such that $d^- = (f_j)^+$ then $G^{f_j,d}$ has a loop on all of its vertices, hence we may apply the $\ell = 0$ case.
Similarly, if $c \in A^*$ such that $c^+ = (f_j)^-$ and $c$ is not a loop then $G_{f_j,c}$ has a loop on all of its vertices hence we may apply the $\ell = 0$ case again. If $c$ is a loop and $j=1$, then now $\mathrm{cdeg}(v)=4$ for all $v\in V$, hence we may apply the first subcase of $\ell = 1$. If $c$ is a loop but $j>1$ then $G_{f_j,c}$ has an undirected path of length $j-1$ without repeated vertices, from $v_1$ to the vertex with corrected degree $3$, hence we may apply the induction hypothesis on $j$. Therefore if $\ell = 1$ then $S(G,B)=0$.

Let $\ell \geq 2$ and let $f_1, \dots, f_j$ be a shortest undirected path of positive length such that the starting node $v_1$ and the ending node $v_0$ have no loops, but every other vertex touched by the path has a loop on it. We prove by induction on $j$. By replacing $G$ by its opposite (if needed), we may assume that $f_j^+ = v_0$. 

Apply Lemma \ref{lem:loops} on $f_j$.
If $d\in A^*$ (resp. $c \in A^*$) such that $d^- = (f_j)^+$ (resp. $c^+ = (f_j)^-$ and $c$ is not a loop) then $G^{f_j,d}$ (resp. $G_{f_j,c}$) has one more vertex with a loop, hence we may apply the induction hypothesis on $\ell$. If $c$ is a loop (and hence $j>1$) then $G_{f_j,c}$ has an undirected path of length $j-1$ without repeated vertices, from $v_1$ to another vertex without a loop, hence we may apply the induction hypothesis on $j$. This completes the proof of the proposition.
\end{proof}

\begin{rem}
For $m\geq 4$ the argument cannot be repeated without significant modifications, since we have no bound on the degrees after the degree-homogenization step, in particular the argument of Lemma \ref{lem:conn} on connected components does not survive.
\end{rem}

\medskip{}


\newpage

\appendix
\section{Proof of Prop. \ref{prop:equivalence}}\label{appendix}

In the proof we need some basic facts about $E^m$ and $M_n E^m$.
It is well known that the $R$-algebra $E^{m}$ has a standard $R$-basis
of size $2^{m}$ consisting of ordered monomials $v_{i_{1}}v_{i_{2}}\dots v_{i_{d}}$
for any $0\leq d\leq m$ and $1\leq i_{1}<i_{2}<\dots<i_{d}\leq m$. Also, note
that $E^{m}$ is a graded-commutative $R$-algebra via the grading 
$\deg(v_{i_{1}}v_{i_{2}}\dots v_{i_{d}})=d$,
i.e. 
\begin{equation}\label{eq:graded-comm}
v_{i_{1}}v_{i_{2}}\dots v_{i_{d}}\cdot v_{j_{1}}v_{j_{2}}\dots v_{j_{e}}=(-1)^{de}v_{j_{1}}v_{j_{2}}\dots v_{j_{e}}\cdot v_{i_{1}}v_{i_{2}}\dots v_{i_{d}}
\end{equation}
for any two basis elements. In particular, elements of even degree
are central.

Consequently, $M_{n}E^{m}$ is also a graded $R$-algebra with the
$R$-basis $\mathcal{B}_{n,m}$ of size $2^{m}n^{2}$ consisting
of the matrix units $E_{\alpha,\beta}\in M_{n}R$ $(1\leq\alpha,\beta\leq n)$
which are of degree zero, multiplied in all possible ways with the basis elements
$v_{i_{1}}v_{i_{2}}\dots v_{i_{d}}\in E^m \subseteq M_n E^m$, which are of degree $d$.

\begin{proof}[Proof of Prop. \ref{prop:equivalence}]
First $s_{k}$ is multilinear in all variables, hence it is a polynomial
identity on $M_{n}E^{m}$ if and only if it is a polynomial identity for any choice of $x_{1},\dots,x_{k}\in\mathcal{B}_{n,m}$. Moreover, $s_k$ is skew-symmetric, hence changing the order of $x_1,\dots,x_k$ may not affect whether $s_k$ vanishes. We may also assume without loss of generality that the $x_j$'s are distinct.

We call a subset $X\subseteq M_nE^m$ \emph{simplified} if for some $\ell\leq m$ there is a (unique) injection $c:[\ell]\to X$ such that for each $x\in X$
\begin{equation}\label{eq:simplified}
x=
\begin{cases}
v_i E_{\alpha_x,\beta_x}&\textrm{if }c(i)=x\\
E_{\alpha_x,\beta_x}&\textrm{if }x\notin c([\ell])
\end{cases}
\end{equation}
for some $1\leq \alpha_x,\beta_x\leq n$ for each $x\in X$.
We prove that $s_k(x_1,\dots,x_k)=0$ holds for any $x_1,\dots,x_k\in \mathcal{B}_{n,m}$ if and only if it holds for any simplified subset $X$ (enumerated arbitrarily) of size $k$.

Let $x_{1}=v_{i_{1}}v_{i_{2}}\dots v_{i_{d}}E_{\alpha,\beta}$
and $x_{2}=v_{j_{1}}v_{j_{2}}\dots v_{j_{e}}E_{\gamma,\delta}$ be standard
basis elements of $M_n E^m$. If $i_{u}=j_{v}$ for some $u\leq d$, $v\leq e$
then $x_{1}yx_{2}=0=x_{2}yx_{1}$ for any $y\in M_{n}E^{m}$, hence
$s_{k}(x_{1},\dots,x_{k})=0$ is automatic. Therefore we may assume
that the $E^m$-coefficients of the basis elements $x_{1},\dots,x_{k}$
are products of disjoint subsets of the generators $v_{1},\dots,v_{m}$. 

Let $x_{1}$ be as above. If $d$ is even 
then by Eq. \ref{eq:graded-comm}, $v_{i_{1}}v_{i_{2}}\dots v_{i_{d}}$
commutes with $v_{j_{1}}v_{j_{2}}\dots v_{j_{e}}$ for any $j_1,\dots,j_e$, hence 
\[
s_{k}(x_{1},\dots,x_{k})=v_{i_{1}}v_{i_{2}}\dots v_{i_{d}}\cdot s_{k}(E_{\alpha,\beta},\dots,x_{k}).
\]
Similarly, if $d$ is odd then 
\[
s_{k}(x_{1},\dots,x_{k})=v_{i_{2}}\dots v_{i_{d}}\cdot s_{k}(v_{i_{1}}E_{\alpha,\beta},\dots,x_{k}).
\]
Hence every $x \in X$ can be reduced to the form $x = v_{i_x}E_{\alpha_x,\beta_x}$ for some distinct $1\leq i_x\leq m$. After relabeling the generators of $E^m$, we may indeed assume that $X$ is simplified.

\begin{claimno} There is a bijection between simplified subsets $X\subseteq M_n E^m$ of size $k$ with a choice of $s,t\in[n]$ and the set of doubly-rooted digraphs $G$ on $V(G)=[n]$ with $|A(G)|=k$ together with an injection $d:[\ell] \to A(G)$ for some $\ell \leq m$ (and without parallel edges outside of $d([\ell])$) such that for a bijective pair the following holds: for any enumeration $X=\{x_1,\dots,x_k\}$ and $B=d([\ell])$,
\begin{equation}\label{eq:connection}
(s_k(x_1,\dots,x_k))_{s,t}= \pm T(G,B) v_{1}\dots v_{\ell}
\end{equation}
for some choice of sign.
\end{claimno}

Note that to have a bijection we have to consider two doubly-rooted digraphs as the same if they only differ by the labeling of their edges.

\begin{proof}[Proof of the claim] 
Given a simplified subset $X$, define the corresponding doubly-rooted digraph $G$ on $V=[n]$ where the set of edges with source $\alpha$ and target $\beta$ is
\[
A(\alpha,\beta):=\{x\in X\,\mid \,\alpha_x=\alpha,\,\beta_x=\beta\}
\]
and $d:=c$, where $c$ is given via Eq. \ref{eq:simplified} in the definition of simplified sets.

Then choose an enumeration $X=\{x_1,\dots,x_k\}$ such that $x_j=c(j)$ for $j\leq \ell$. In other words,
\[x_j=
\begin{cases}
v_j E_{\alpha_{x_j},\beta_{x_j}}&\textrm{if } 1 \leq j\leq \ell \\
E_{\alpha_{x_j},\beta_{x_j}}&\textrm{if }\ell < j \leq k.
\end{cases}
\]
For simplicity, we write $\alpha_j$ instead of $\alpha_{x_j}$ for all $1\leq j\leq k$. Using that the $v_i$'s commute with the matrix units and $M=[\ell]$, we may compute
\[
s_{k}(x_{1},\dots,x_{k})=\sum_{\sigma\in\mathfrak{S}_{k}}
\mathrm{sgn}(\sigma) x_{\sigma(1)}\dots x_{\sigma(k)}=
\]
\[
=\sum_{\sigma\in\mathfrak{S}_{k}}\mathrm{sgn}(\sigma)
(v_{\sigma_M(1)}\dots v_{\sigma_M(\ell)}) E_{\alpha_{\sigma(1)},\beta_{\sigma(1)}}\dots 
E_{\alpha_{\sigma(k)},\beta_{\sigma(k)}}=
\]
\[
=(v_1\dots v_\ell)\sum_{\sigma\in\mathfrak{S}_{k}}
\mathrm{sgn}(\sigma)\mathrm{sgn}(\sigma_M)
E_{\alpha_{\sigma(1)},\beta_{\sigma(1)}}\dots 
E_{\alpha_{\sigma(k)},\beta_{\sigma(k)}}
\]
where 
\[
E_{\alpha_{\sigma(1)},\beta_{\sigma(1)}}\dots 
E_{\alpha_{\sigma(k)},\beta_{\sigma(k)}} = 
E_{s,t}
\]
if $\alpha_{\sigma(1)}=s$, $\beta_{\sigma(k)}=t$ and $\alpha_{\sigma(j+1)}=\beta_{\sigma(j)}$ for all $j=1,\dots,k-1$, i.e. when $(x_{\sigma(1)},\dots,x_{\sigma(k)})$ represents an Eulerian trail in $G$ from $s$ to $t$, otherwise the product is $0$. The sign in Eq. \ref{eq:connection} may appear for a different enumeration of $X$.

Conversely, given $G$ and $d:[\ell]\to A$ of the above form, for each $a\in A$ we may define $x_a$ as in Eq. \ref{eq:simplified} with $\alpha_{x_a}=a^-$, $\beta_{x_a}=a^+$ and $c(i)=x_{d(i)}$ for all $i\in [\ell]$. These two constructions are inverses of each other and hence the claim follows.
\end{proof}

The proposition follows directly from the claim, using that $T(G,B)$ is independent of the enumeration $d:[\ell]\to B$.
\end{proof}


\begin{thebibliography}{}
\bibitem[AL]{AL} A. S. Amitsur and J. Levitzki, Minimal identities
for algebras, Proc. Amer. Math. Soc. 1 (1950), 449\textendash 463.

\bibitem[D]{D}M. Domokos, Eulerian polynomial identities and algebras
satisfying a standard identity, J. Algebra 169 (1994), 913\textendash 928.

\bibitem[F]{F}P.E. Frenkel, Polynomial identities for matrices over the Grassmann algebra, 
Israel J. Math. 220 (2017) 791\textendash 801.

\bibitem[MMSzW]{MMSzW}L. M\'arki, J. Meyer, J. Szigeti and L. van Wyk,
Matrix representations of finitely generated Grassmann algebras and
some consequences, Israel J. Math. 208 (2015), 373\textendash 384.

\bibitem[S1]{S1}R. G. Swan, An application of graph theory to algebra, 
Proc. Amer. Math. Soc. 14 (1963), 367\textendash 373.

\bibitem[S2]{S2}R. G. Swan, Correction to "An application of graph theory to algebra", 
Proc. Amer. Math. Soc. 21 (1969), 379\textendash 380.

\end{thebibliography}
\end{document}